\documentclass{article}
\usepackage[utf8]{inputenc} 
\usepackage{lipsum}
\usepackage{amsmath}   
\usepackage{amssymb}
\usepackage{graphicx}
\usepackage{authblk} 

\usepackage{indentfirst,csquotes}

\usepackage{amssymb,amsthm,amsmath}
\usepackage{xcolor,paralist,hyperref,titlesec,fancyhdr,etoolbox}
\newtheorem{theorem}{Theorem}[]
\newtheorem{definition}[theorem]{Definition}
\newtheorem{example}{Example}[section]
\newtheorem{lemma}[theorem]{Lemma}
\newtheorem{proposition}[theorem]{Proposition}
\newtheorem{corollary}[theorem]{Corollary}

\titleformat{\section}
  {\normalfont\large\centering} 
  {\thesection}                        
  {10pt}                               
  {}                                  

\hypersetup{ colorlinks=true, linkcolor=black, filecolor=black, urlcolor=black }

\title{Abhyankar-Moh Semigroups for arbitrary hypersurfaces} 

\author[1]{Fuensanta Aroca}
\author[1]{Annel Ayala}
\author[2]{Oscar Castañon}
\author[3]{Giovanna Ilardi}

\affil[1]{\small Instituto de Matemáticas, Unidad Oaxaca, Universidad Nacional Autónoma de México, Oaxaca, Oaxaca, México, \texttt{fuen@im.unam.mx}, \texttt{annelayv@gmail.com}}

\affil[2]{\small Escuela Superior de Matemáticas, Número 3, Universidad Autónoma de Guerrero, Iguala, Guerrero, México, \texttt{oscar.moreno@im.unam.mx}}

\affil[3]{\small Dipartimento Matematica Ed Applicazioni ``R. Caccioppoli''
Universit\`{a} Degli Studi Di Napoli ``Federico II'' Via Cintia -
Complesso Universitario Di Monte S. Angelo 80126 - Napoli - Italia,
\texttt{giovanna.ilardi@unina.it}}

\date{\today}

\begin{document}

\maketitle

\begin{abstract}
For an arbitrary hypersurface singularity, we construct a family of semigroups associated with algebraically closed fields that arise as an infinite union of rings of series. These semigroups extend the value semigroup of a plane curve studied by S. S. Abhyankar and T. T. Moh \cite{ AbhyankarSing:1977, Abhyankar:1977, AbhyankarMoh:1973}. The algebraically closed fields under consideration possess a natural valuation that induces a corresponding value semigroup. We establish the necessary conditions under which these semigroups are independent of the choice of the root. Moreover, the extensions proposed by P. D. González, K.-H. Kiyek and M. Micus  \cite{Gonzalez:2003, KiyekMicus:1988}, where they specifically address the case of quasi-ordinary singularities, and the extensions introduced by A. Sathaye \cite{Sathaye:1983} and by A. Ali and A. Assi  \cite{AbbasAssi:2021}, can be understood as particular instances within our constructed family.

 MSC [2000]: {05E40; 20M14; 13F25}.
 
 Keywords: Semigroup; Formal power series; Newton-Puiseux expansions

\end{abstract}

\maketitle

\let\thefootnote\relax

\section{Introduction}

Let $K$ be an algebraically closed field of characteristic zero. Let $K[[x]]$ denote the ring of formal power series in $x$ over $K$ and let $K((x))$ be its field of fractions.

The Newton-Puiseux theorem \cite{Puiseux:1850} asserts that,  
if $f\in K[[x]][y]$ is a monic irreducible polynomial  of degree $d$, then  it factors in $K[[x^\frac{1}{d}]][y]$ as
\begin{equation}\label{factorizacion curva}
   f= \prod_{\eta^d=1} (y-\xi (\eta x^\frac{1}{d}))
\end{equation}
 where $\xi \in K[[x^\frac{1}{d}]]$.

 Given a series $\xi=\displaystyle\sum_{i=\alpha}^\infty a_i x^\frac{i}{d}\in K((x^\frac{1}{d}))$ we denote
\begin{equation}\label{orden de serie}
 ord_x (\xi) := \min_{a_i\neq 0} \frac{i}{d}.
\end{equation}

 In 1973, S. S. Abhyankar and T. T. Moh \cite{AbhyankarMoh:1973,Abhyankar:1977,AbhyankarSing:1977}
 studied the structure of the semigroup
\begin{equation}\label{Semigrupo Curva Plana}
   \Gamma:= \{ ord_x h(\xi) \ ; \ h\in K[[x]][y]\setminus (f)\}.
\end{equation}

In 1948, C. Arf \cite{Arf} had already introduced a semigroup similar to the one in (\ref{Semigrupo Curva Plana}). P. Du Val \cite{Duval} discussed Arf's semigroup results, providing an alternative interpretation. A more recent discussion on the study of semigroups on hypersurfaces can be found in \cite{Patrick2004}.

As a consequence of (\ref{factorizacion curva}), the semigroup $\Gamma$ does not depend on the chosen root of $f$ and it makes sense to say that $\Gamma$ is the {\bf value semigroup of the plane curve defined by $f$}. 

It is worth noting that the semigroup of a plane branch is classically defined in terms of intersection multiplicities. The semigroup obtained in this way is contained in $\mathbb{Z}$ and equals $\Gamma$ up to rescaling.

The value semigroup is a useful tool to study and classify plane curve singularities (see for example \cite{Zariski:1986}). In particular, it determines the topological type of the singularity. Moreover, the structure of this semigroup is a useful tool in coding theory (see for example \cite{CampilloFarran:2000, GalindoMonserrat:2009}).

To extend the concept of ``value semigroup'' to hypersurfaces, the following are needed:
\begin{itemize}
    \item A suitable field $\mathcal{S}$ containing $K[[x_1,\cdots ,x_n]]$ such that $f\in K[[x_1,\cdots ,x_n]][y]$ splits as an element of  $\mathcal{S}[y]$  (so that we can choose $\xi$).
    \item  A mapping $\nu:\mathcal{S} \longrightarrow \mathbb{Q}^n $ analogous to the mapping $ord_x: K((x^{\frac{1}{d}}))\longrightarrow \mathbb{Q}$.
    \item A subring $\mathcal{A}\subset \mathcal{S}$  to consider the values of $h(\xi)$ with $h\in\mathcal{A}[y]$.
\end{itemize}

 In 1983, Abhyankar's student, A. Sathaye \cite{Sathaye:1983}, gave a generalization of Abhyankar-Moh results when $x$ is replaced by an $n$-tuple $(x_1,\ldots ,x_n)$. Sathaye's definition uses the field of iterated Puiseux series as $\mathcal{S}$, the minimum of the support with the rev-lex order as $\nu$, and the ring $K[[x_1,\ldots ,x_n]]$ as $\mathcal{A}$.  P. D. González, K.-H. Kiyek and M. Micus gave, independently, an extension for quasi-ordinary singularities \cite{Gonzalez:2003, KiyekMicus:1988}. González's construction has been extended to $\sigma$-free singularities by A. Ali and A. Assi  \cite{AbbasAssi:2021} using the ideas of J. M. Tornero,  presented in \cite{Tornero:2003,Tornero:2008}.

 In this paper we construct, for an arbitrary hypersurface singularity, a family of semigroups, defined in terms of the family of algebraically closed fields constructed in \cite{ArocaIlardi:2009}. These  algebraically closed fields have a natural valuation that induces a value semigroup. We give the necessary conditions so that these semigroups do not depend on the chosen root $\xi$.

The constructions introduced by A. Sathaye, by P. D. González, and jointly by A. Ali and A. Assi naturally arise as specific examples within our defined family.

  Some of the results that we present in this article are also presented in  \cite{Aroca:Ayala}.

\section{A family of algebraically closed fields}

As we pointed out in the introduction, to extend the concept of the Abhyankar-Moh semigroup from plane curves to hypersurfaces $V(f)$, we need to be able to produce a root of $f\in K[[x_1,\ldots ,x_n]][y]$. 

Given $f\in K[[x_1,\ldots, x_n]][y]$, let $\Delta_y (f)$ denote the discriminant of $f$ with respect to $y$. The hypersurface singularity $V(f)$ is said to be \emph{quasi-ordinary} when
$$\Delta_y (f)= X^\gamma u,$$
with $\gamma\in (\mathbb{Z}_{\geq 0})^n$ and $u$ a unit in $K[[{x_1}, \ldots , {x_n}]]$.

González's extension uses the Abhyankar-Jung Theorem \cite{Abhyankar:1958}, which guarantees that, for quasi-ordinary singularities, the roots belong to $ K[[{x_1}^\frac{1}{d},\ldots ,{x_n}^\frac{1}{d}]]$. To extend González's construction, A. Ali and A. Assi use a theorem due to J. McDonald \cite{McDonald:1995} that ensures the existence of a cone $\sigma$ such that $f$ splits in $K_\sigma [[{x_1}^\frac{1}{d},\ldots ,{x_n}^\frac{1}{d}]]$.  Sathaye's construction relies on the fact that the field of iterated Puiseux series is algebraically closed.

McDonald’s theorem was originally stated for polynomials in several variables in the ring $K[x_1,\ldots, x_n][y]$. It was generalized to polynomials in $K[[x_1, \ldots , x_n]][y]$, and in the complex analytic case this generalization was connected to the Newton polyhedron of the discriminant of $f$ with respect to $y$ in \cite{Gonzalez:2000}.

In this section we recall the construction of a family of algebraically closed fields presented in \cite{ArocaIlardi:2009,ArocaRond:2019}. For a detailed discussion of the fields $K_\sigma ((X))$ and $K_{\preceq}((X))$ we refer the reader to the beautiful article written by A. A. Monforte and M. Kauers  \cite{MonforteKauers:2013}.

Let $\preceq$ be an additive total order on $\mathbb{R}^n$ (compatible with the group structure). 
We denote by ${(\mathbb{R}_{\geq 0})}^n$ the first orthant of $\mathbb{R}^n$, that is,
${(\mathbb{R}_{\geq 0})}^n:=\{(x_1,\dots,x_n)\in\mathbb{R}^n \  ; \  x_i\geq 0,\ \text{for } i=1,\dots,n\}$.
We also write
$(\mathbb{R}^n)_{\succeq \underline{0}}:=\{u\in\mathbb{R}^n \ ; \  \underline{0}\preceq u\}$.

A subset  $\sigma \subset \mathbb{R}^n$ is a (convex polyhedral rational) \textbf{cone}  when $$\sigma=<u_1, u_2, \ldots, u_s>=\{\lambda_1u_1+\lambda_2u_2 +\ldots+\lambda_su_s\hspace{0.1cm} ; \hspace{0.1cm} \lambda_i\in \mathbb{R}_{\geq 0}\}$$ for some $u_1, u_2,\ldots , u_s \in \mathbb{Q}^n$.

If $\sigma\subset (\mathbb{R}^n)_{\succeq \underline{0}}$ then $\sigma$ does not contain any nontrivial linear subspace and the set of formal series

\begin{center}
$K_\sigma [[X]]:=\left\{\displaystyle\sum_{\gamma \in \sigma \cap \mathbb{Z}^n}a_\gamma X^{\gamma} \ ; \ a_\gamma \in K \right\}$
\end{center}
has a natural ring structure, where $X=(x_1,\dots,x_n)$. The ring $K_\sigma [[X]]$ is the completion of the coordinate ring of an affine toric variety with respect to a certain maximal ideal. We denote by $K_\sigma((X))$ its field of fractions.

A cone $\sigma$ and an order $\preceq$ are said to be \textbf{compatible} when $\sigma \subseteq (\mathbb{R}^n)_{\succeq \underline{0}}$. In what follows, the symbol $\preceq$ will stand for an additive  total order on $\mathbb{R}^n$ compatible with the first orthant. We will work with cones compatible with $\preceq$ containing the first orthant. That is ${(\mathbb{R}_{\geq 0})}^n\subset \sigma\subset (\mathbb{R}^n)_{\succeq \underline{0}}$. 

We will be using the following rings $K[[X]]\subset K_{\preceq}[[X]]\subset K_{\preceq}[[{X}^{\frac{1}{k}}]] \subset K_\preceq[[X^*]]$,

$$K_{\preceq}[[X]]:= \bigcup_{\sigma \subset (\mathbb{R}^n)_{\succeq \underline{0}}} K_\sigma [[X]],$$
$$\text{for}\, k\in \mathbb{Z}_{\geq 1},\quad K_{\preceq}[[{X}^{\frac{1}{k}}]] := K_{\preceq}[[{x_1}^\frac{1}{k}, \ldots , {x_n}^\frac{1}{k} ]],$$
$$K_\preceq[[X^*]] :=\bigcup_{k\in \mathbb{Z}_{\geq 1}}K_{\preceq}[[{X}^{\frac{1}{k}}]],$$ 
and their corresponding fields of fractions $K((X))\subset K_{\preceq}((X)) \subset K_{\preceq}((X^{\frac{1}{k}})) \subset K_{\preceq}((X^*))$, see, for example, \cite[p.~358]{MonforteKauers:2013},

\begin{equation}
    K_{\preceq}((X)):=\{\varphi \ ; \  \exists \ \gamma \in \mathbb{Z}^n,\  X^\gamma\varphi \in K_{\preceq}[[X]]\}= \bigcup_{\sigma \subset (\mathbb{R}^n)_{\succeq \underline{0}}} K_\sigma ((X)),
\end{equation}

\begin{equation}\label{union con ramificacion}
    K_{\preceq}((X^{\frac{1}{k}})):=\{\varphi \ ; \  \exists \ \gamma \in \mathbb{Z}^n, \  \exists \ k \in \mathbb{Z}_{>0},\  X^{ \frac{\gamma}{k}}\varphi \in K_{\preceq}[[{X}^{\frac{1}{k}}]]\} = \bigcup_{\sigma \subset (\mathbb{R}^n)_{\succeq \underline{0}}} K_\sigma ((X^{\frac{1}{k}})),
\end{equation}

\begin{equation}\label{algebraicamente cerrado}
    K_{\preceq}((X^*)):=\{\varphi \ ; \  \exists\  \gamma \in \mathbb{Z}^n,\   \exists\  k \in \mathbb{Z}_{>0}, \  X^{ \frac{\gamma}{k}}\varphi \in K_\preceq[[X^*]]\}=\bigcup_{k \in \mathbb{Z}_{\geq 1}} K_\preceq ((X^{\frac{1}{k}})).
\end{equation}

Let $R=K_\preceq[[X^*]], K_{\preceq}[[X]]$ or $ K_{\preceq}[[{X}^{\frac{1}{k}}]]$, an element $\displaystyle\sum_{\gamma \in \Lambda}a_\gamma X^{\gamma}\in R$  is a unit of $R$ if and only if $a_{(0,\ldots 0)}\neq 0$  (see, for example, Theorem 12 \cite{MonforteKauers:2013}).

Given a series $\varphi=\displaystyle\sum_{\gamma \in \Lambda}a_\gamma x^{\gamma}$, the support of $\varphi$ is the set $Supp(\varphi):= \{\gamma \in \Lambda \ ; \  a_\gamma\neq 0\}$.

When $K$ is an algebraically closed field of characteristic zero, the field $K_{\preceq}((X^*))$ is algebraically closed \cite[Theorem 1, Theorem 4.5]{ArocaIlardi:2009,ArocaRond:2019}.

Note that, in several variables, the ring of Puiseux power series (as defined, for example, in \cite{Tornero:2008}) is contained in $K_{\preceq}((X^*))$ if and only if  the first orthant is compatible with $\preceq$. Therefore, by the Abhyankar-Jung Theorem \cite{Abhyankar:1958}, the roots of a quasi-ordinary polynomial in $K_{\preceq}((X^*))$ will coincide for any order $\preceq$ with ${\left(\mathbb{R}_{\geq 0}\right)}^n\subset {\left(\mathbb{R}^n\right)}_{\succeq \underline{0}}$. The same applies for Puiseux hypersurfaces, i.e.   hypersurfaces that can be locally parametrized by Puiseux series.

Given a vector $\omega\in {\left(\mathbb{R}_{>0}\right)}^n$ of rationally independent coordinates, $\omega$ induces a total order  on $\mathbb{Q}^n$ compatible with the group structure given by

\begin{equation}\label{9811111}
\alpha \preceq_{\omega} \beta \hspace{0.2 cm}\text{if and only if} \hspace{0.2 cm}\omega\cdot \alpha \leq \omega \cdot \beta.
\end{equation}
The order $\preceq_{\omega}$ may be extended to a total order $\preceq$ on $\mathbb{R}^n$ \cite{Robbiano:1986}. The field $K_{\preceq}((X^*))$ is the same, independently of the extension and we may denote $K_{\preceq_{\omega}}((X^*)):=K_{\preceq}((X^*))$.

M. Buchacher \cite{Buchacher:2023} has implemented an algorithm using \textit{Mathematica} that computes the first terms of the roots in $K_{\preceq_\omega}((X^*))$ of polynomials $f\in K[X][y]$. We have used his implementation for the examples presented in this paper.

\begin{example}\label{essd}
Set 
$f(z)=z^2-2x_1^2x_2^4\,z+\left(x_1^4x_2^8-x_1x_2^3-2x_1^2x_2^4+6x_1^2x_2^5-x_1^3x_2^5+6x_1^3x_2^6-9x_1^3x_2^7\right) \in \mathbb{C}[[x_1, x_2]][z]$.

We have that $\Delta_z(f)= 4x_1x_2^3\left(1+x_1x_2-3x_1x_2^2\right)^2$, hence $f$ is quasi-ordinary, its roots in $\mathbb{C}[[x_1^{\frac{1}{2}}, x_2^{\frac{1}{2}}]]$ are

$$
\xi_1=x_1^{\frac{1}{2}}x_2^{\frac{3}{2}}+x_1^{\frac{3}{2}}x_2^{\frac{5}{2}}-3x_1^{\frac{3}{2}}x_2^{\frac{7}{2}}+x_1^2x_2^4,
\hspace{0.3cm}
\xi_2=-x_1^{\frac{1}{2}}x_2^{\frac{3}{2}}-x_1^{\frac{3}{2}}x_2^{\frac{5}{2}}+3x_1^{\frac{3}{2}}x_2^{\frac{7}{2}}+x_1^2x_2^4.
$$
Note that $\xi_1$ and $\xi_2$ belong to $\mathbb{C}_{\preceq}((X^*))$ for any order $\preceq$ compatible with the first orthant.
\end{example}

\begin{example}\label{23231s}
Set $f(y)=y^4 - 2(x_1+x_2)y^2 + x_1^2 - 2x_1x_2 + x_2^2 \in \mathbb{C}[[x_1,x_2]][y]$. The roots of $f$ are: $$y=\pm x_1^{\frac{1}{2}}\pm x_2^{\frac{1}{2}},$$ and they belong to $\mathbb{C}[[x_1^{\frac{1}{2}},x_2^{\frac{1}{2}}]]$, hence $f$ defines a Puiseux hypersurface. 
Note that $\Delta_y(f)=4096 {x_1}^2 {x_2}^2 ({x_1}^2 - 2x_1x_2 + {x_2}^2)$, so $f$ is not quasi-ordinary and $\pm x_1^{\frac{1}{2}}\pm x_2^{\frac{1}{2}}$ belong to $\mathbb{C}_{\preceq}((X^*))$ for any order $\preceq$ compatible with the first orthant.

\end{example}

\begin{example}\label{ejemplo3}
    The roots of $f(z)= z^2-(x+y^2) \in \mathbb{C}[[x,y]][z]$, in the field $\mathbb{C}_{\preceq_\omega}((X^*))$ with $\omega :=(4, \sqrt{2})$ are 

$$\xi_1=-y-\frac{1}{2}xy^{-1}+\frac{1}{8}x^2y^{-3}-\frac{1}{16}x^3y^{-5}+\frac{5}{128}x^4y^{-7}-\frac{7}{256}x^5y^{-9}+\cdots$$
   
    and 
  
$$\xi_2=y+\frac{1}{2}xy^{-1}-\frac{1}{8}x^2y^{-3}+\frac{1}{16}x^3y^{-5}-\frac{5}{128}x^4y^{-7}+\frac{7}{256}x^5y^{-9}+\cdots.$$
Taking $\omega :=(1, \sqrt{2})$ the roots of $f$  in $\mathbb{C}_{\preceq_{\omega}}((X^*))$ are:

$$\Tilde{\xi}_1 =-x^{\frac{1}{2}}-\frac{1}{2}x^{-\frac{1}{2}}y^2+\frac{1}{8}x^{-\frac{3}{2}}y^4-\frac{1}{16}x^{-\frac{5}{2}}y^6+\frac{5}{128}x^{-\frac{7}{2}}y^8-\frac{7}{256}x^{-\frac{9}{2}}y^{10}+\cdots$$ 
and
$$\Tilde{\xi}_2 =x^{\frac{1}{2}}+\frac{1}{2}x^{-\frac{1}{2}}y^2-\frac{1}{8}x^{-\frac{3}{2}}y^4+\frac{1}{16}x^{-\frac{5}{2}}y^6-\frac{5}{128}x^{-\frac{7}{2}}y^8+\frac{7}{256}x^{-\frac{9}{2}}y^{10}+\cdots.$$

Let $\eta := \langle (0,1), (1,-2)\rangle$ and $\Tilde{\eta} = \langle (1,0), (-1,2)\rangle$. We have $\xi_1, \xi_2\in \mathbb{C}_{\eta}((x,y))$ and $\Tilde{\xi}_1, \Tilde{\xi}_2 \in \mathbb{C}_{\Tilde{\eta}}((x^{\frac{1}{2}},y^\frac{1}{2}))$. Previously, we established that every order $\preceq$ under consideration is compatible with the first orthant. Therefore, $(0,0)\preceq(1,0)$ and $(0,0)\preceq(0,1)$. Consequently, $\preceq$ is compatible with either $\eta$ or $\widetilde{\eta}$. Indeed, if $(1,-2)\prec(0,0)$, then, by the additivity of $\preceq$, adding $(-1,2)$ to both sides yields $(0,0)\prec(-1,2)$. Similarly, if $(-1,2)\prec(0,0)$, adding $(1,-2)$ to both sides gives $(0,0)\prec(1,-2)$. Therefore, $(0,0)\preceq(1,-2)$ or $(0,0)\preceq(-1,2)$, and hence $\preceq$ is compatible with either $\eta$ or $\widetilde{\eta}$.

When $\preceq$ is compatible with $\eta$, its roots in $\mathbb{C}_{\preceq}((X^*))$ are $\xi_1$ and $\xi_2$ and, when $\preceq$ is compatible with $\Tilde{\eta}$, its roots in $\mathbb{C}_{\preceq}((X^*))$ are $\Tilde{\xi}_1$ and  $\Tilde{\xi}_2$.
\end{example}

\section{Dominating exponent}

Given a series $\xi\in K ((x))$ and $ord_x$ as in (\ref{orden de serie}) we have that
\begin{equation}
    \xi= x^{ord_x(\xi)}\Bar{\xi}
\end{equation}
where $\Bar{\xi}= x^{-ord_x(\xi)}\xi$ is a unit of $K[[x]]$. The term $ax^{ord_x(\xi)}$ in the series $\xi$ is called the \textbf{dominating term}. The exponent of the dominating term is called the \textbf{dominating exponent}. 

The semigroup (\ref{Semigrupo Curva Plana}) of a plane curve  is the semigroup of the exponents of the dominating terms of elements of $K[[x]][y]$ evaluated on a root of its defining polynomial.

We quote P. Popescu-Pampu \cite{MR1911654}:

"A difficulty for extending the plane branch definition of the semigroup is that in dimension $> 1$, fractional
series may have no dominating term. \emph{One way to force the existence of a dominating term is to restrict to those functions which do have one.}"

Instead of restricting the ring of functions to those that have a dominating exponent, our approach is to enlarge the ring of functions so that every element can be written as in (\ref{monomio por unidad}) in a unique way.

\begin{lemma}\label{5123812}
Given an element   $\xi \in K_{\preceq}((X^*))$, $\xi \neq 0$, there exists a unique $\gamma\in\mathbb{Q}^n$ such that
\begin{equation}\label{monomio por unidad}
\xi=X^\gamma u,
\end{equation}
with $u \in K_\preceq[[X^*]]$ invertible.
\end{lemma}
\begin{proof}
Let $\xi=\displaystyle\sum_\alpha a_\alpha X^\alpha$ and set $\gamma:= \min_\preceq \{\alpha \  ;\  a_\alpha\neq 0\}$. 

\begin{equation}\label{unidad por monomio}
    \xi=\sum_\alpha a_\alpha X^\alpha=X^\gamma\left(\displaystyle\sum_\alpha a_\alpha X^{\alpha-\gamma}\right).
\end{equation}

Now $\phi:= \displaystyle\sum_\alpha a_\alpha X^{\alpha-\gamma}=\displaystyle\sum_\alpha b_\alpha X^\alpha$ where $b_\alpha = a_{\alpha +\gamma}$. Since $b_{(0,\ldots ,0)} = a_\gamma\neq 0$, we have that $\phi$ is a unit.

\end{proof}

To make explicit the ring we are working on, we will say that $\gamma$ as in Lemma \ref{5123812} is the $\preceq$-\textbf{dominating exponent} of $\xi$. When the order $\preceq$ is an extension of $\preceq_\omega$, the $n$-tuple $\gamma$ will be called the $\omega$-\textbf{dominating exponent}.

\begin{example}
Set $\xi_1=x_1^{\frac{1}{2}}x_2^{\frac{3}{2}}+x_1^{\frac{3}{2}}x_2^{\frac{5}{2}}-3x_1^{\frac{3}{2}}x_2^{\frac{7}{2}}+x_1^2x_2^4 \in \mathbb{C}[[x_1^{\frac{1}{2}},x_2^{\frac{1}{2}}]]$. The $\preceq$-dominating exponent of $\xi$ is $(\frac{1}{2},\frac{3}{2})$ for any order $\preceq$ such that the first orthant is compatible with $\preceq$.
\end{example}

\begin{example}
Set $\xi :=  x_1^{\frac{1}{2}} + x_2^{\frac{1}{2}} \in \mathbb{C}[[x_1^{\frac{1}{2}},x_2^{\frac{1}{2}}]]$. 

The $(\sqrt{2},1)$-dominating exponent of $\xi$ is $(0,\frac{1}{2})$ whilst the  $(1,\sqrt{2})$-dominating exponent of $\xi$ is $(\frac{1}{2}, 0)$.

\end{example}

\begin{example}
    The $(1, \sqrt{2})$-dominating exponent of 
    $$\Tilde{\xi}_1 = -x^\frac{1}{2}+\frac{1}{2}x^{-\frac{1}{2}}y^2+\frac{1}{8}x^{-\frac{3}{2}}y^4+\frac{1}{16}x^{-\frac{5}{2}}y^{6}+\frac{5}{128}x^{-\frac{7}{2}}y^8+\frac{7}{256}x^{-\frac{9}{2}}y^{10}+\cdots \in \mathbb{C}_{\preceq_{(1,\sqrt{2})}}((x^{\frac{1}{2}},y^{\frac{1}{2}}))$$
    is $(\frac{1}{2},0)$.
   
    The  $(4, \sqrt{2})$-dominating exponent of 
$$\xi_1 =-y-\frac{1}{2}xy^{-1}+\frac{1}{8}x^2y^{-3}-\frac{1}{16}x^3y^{-5}+\frac{5}{128}x^4y^{-7}-\frac{7}{256}x^5y^{-9}+\cdots \in \mathbb{C}_{\preceq_{(4,\sqrt{2})}}((x,y))$$ 
is $(0,1)$. 
\end{example}

\begin{corollary}\label{Es el anillo de valoracion}
    An element $\xi\in K_\preceq((X^*))$ belongs to the ring $K_\preceq[[X^*]]$ if and only if its $\preceq$-dominant exponent belongs to ${\left(\mathbb{R}^n\right)}_{\succeq \underline{0}}$.
\end{corollary}
\begin{proof}
 Let $\gamma$ be the $\preceq$-dominant exponent of $\xi$. By Lemma \ref{5123812} there exists a unit $u \in K_\preceq[[X^*]]$ such that
$\xi=X^\gamma u$. Since $u$ is a unit,  $\underline{0}\in Supp(u)$ and then $\gamma\in Supp(X^\gamma u)\subset {\left(\mathbb{R}^n\right)}_{\succeq \underline{0}}$. If $\gamma\in {\left(\mathbb{R}^n\right)}_{\succeq \underline{0}}$, the monomial $X^\gamma$ is in $K_\preceq[[X^*]]$ and, therefore, the product $X^\gamma u$ belongs to $K_\preceq[[X^*]]$. 
\end{proof}

\section{Order and branches}

We quote S. S. Abhyankar \cite{AbhyankarMoh:1973}.

``If $z(t)$ is any element of $K((t))$ such that $f (t^n, z(t)) = 0$ then $z(t) = y(wt)$ for some
$w \in
\mu_{n}(k)$. 
In particular, we have $Supp_t z(t) = Supp_t y(t)$. Thus the set $Supp_t y(t)$ depends only on $f$ and not on a root $y(t)$ of $f (t^n, Y) = 0$. Therefore we can make ...''

The above property does not hold anymore when we consider several variables: The polynomial $f(y):=y^5+x_1^2x_2^2y^2+x_2^5$ is irreducible as an element of $K[[x_1,x_2]][y]$, the roots of $f$ in $K_{\preceq_{(1, \sqrt{5})}}((X^*))$    are
\begin{equation*}
    \xi_1=-x_1^{\frac{2}{3}}x_2^{\frac{2}{3}}-\frac{1}{3}x_1^{-\frac{8}{3}}x_2^{\frac{7}{3}}+\frac{1}{3}x_1^{-6}x_2^4-\frac{44}{81}x_1^{-\frac{28}{3}}x_2^{\frac{17}{3}}+\cdots,
    \end{equation*}
\begin{equation*}
        \xi_2=\frac{1-i\sqrt{3}}{2}x_1^{\frac{2}{3}}x_2^{\frac{2}{3}}+\frac{2i}{3i+3\sqrt{3}}x_1^{-\frac{8}{3}}x_2^{\frac{7}{3}}+\frac{1}{3}x_1^{-6}x_2^4-\frac{-44i+44\sqrt{3}}{81i+81\sqrt{3}}x_1^{-\frac{28}{3}}x_2^{\frac{17}{3}}+\cdots,
        \end{equation*}

         \begin{equation*}
        \xi_3=\frac{1+i\sqrt{3}}{2}x_1^{\frac{2}{3}}x_2^{\frac{2}{3}}-\frac{2i}{-3i+3\sqrt{3}}x_1^{-\frac{8}{3}}x_2^{\frac{7}{3}}+\frac{1}{3}x_1^{-6}x_2^4-\frac{44i+44\sqrt{3}}{-81i+81\sqrt{3}}x_1^{-\frac{28}{3}}x_2^{\frac{17}{3}}+\cdots,
        \end{equation*}
\begin{equation*}
         \xi_4=-ix_1^{-1}x_2^{\frac{3}{2}}-\frac{1}{2}x_1^{-6}x_2^4+\frac{9i}{8}x_1^{11}x_2^{\frac{13}{2}}+\frac{7}{2}x_1^{-16}x_2^{9}+\cdots,
\end{equation*}
\begin{equation*}
\xi_5=ix_1^{-1}x_2^{\frac{3}{2}}-\frac{1}{2}x_1^{-6}x_2^4-\frac{9i}{8}x_1^{11}x_2^{\frac{13}{2}}+\frac{7}{2}x_1^{-16}x_2^{9}+\cdots,
\end{equation*} 
and they do not have the same support.

 Another example of this phenomenon may be found in the introduction of G. Rond's paper \cite{Rond:2017}.

\begin{lemma}\label{factorizacion de polinomio monico}
    A monic polynomial $f \in K_\preceq [[X]][y]$ splits in $K_\preceq[[X^*]]$.
\end{lemma}
\begin{proof}
Since $K_\preceq((X^*))$ is algebraically closed, $f$ splits in $K_\preceq((X^*))$.  Let $\xi$ be a root of $f$ in $K_\preceq((X^*))$ and let $\gamma$ be the $\preceq$-dominating exponent of $\xi$.

Let \[
f(y)=y^d + a_{d-1}y^{d-1} + \cdots + a_1 y + a_0 \in K_\preceq [[X]][y],
\]
and, for $i\in \{ 0,\ldots ,d-1\}$, let $\eta_i$ be the $\preceq$-dominating exponent of $a_i$.

We have
\begin{equation}\label{razonamiento ordenes}
   f(\xi)=X^{d\gamma}u_d + X^{(d-1)\gamma+\eta_{d-1}}u_{d-1} +\cdots +  X^{\gamma+\eta_1}u_1 +X^{\eta_0}u_0 =0,
\end{equation}
where the $u_i$'s are units in $K_\preceq[[X^*]]$ and the $\eta_i$'s are $\preceq$-nonnegative (by Corollary \ref{Es el anillo de valoracion}).

Suppose that $\gamma\prec\underline{0}$, then for $i\in \{ 0,\ldots ,d-1\}$, $d\gamma\prec i\gamma +\eta_i$ and equation (\ref{razonamiento ordenes}) cannot hold. Then $\gamma\in {\left(\mathbb{R}^n\right)}_{\succeq \underline{0}}$ and, by Corollary \ref{Es el anillo de valoracion}, $\xi$ is in $K_\preceq[[X^*]]$.
\end{proof}

\begin{lemma}\label{irreducibilidad}
    An irreducible monic polynomial in $K_\preceq [[X]][y]$ is also irreducible as an element of $K_\preceq ((X))[y]$.
\end{lemma}
\begin{proof}
    If a monic polynomial $g\in K_\preceq ((X))[y]$ divides a monic polynomial $f\in K_\preceq [[X]][y]$ then the roots of $g$ are a subset of the roots of $f$. This implies, by Lemma \ref{factorizacion de polinomio monico}, that the roots of $g$ are in $K_\preceq[[X^*]]$. Therefore, $g\in K_\preceq[[X^*]][y]\cap K_\preceq ((X))[y]= K_\preceq [[X]][y]$.
\end{proof}

From now on we will work with monic polynomials in $K[[X]][y]$ and their roots in $K_\preceq[[X^*]]$.

    Let $f \in K[[X]][y]$ be a reduced monic polynomial and let $g_1, \ldots, g_l \in K_\preceq[[X]][y]$ be irreducible with  $f=g_1\cdots g_l$. Since the extension $K_{\preceq}((X^*))/K((X))$ is separable, the set of roots of $f$ in $K_{\preceq}((X^*))$ is the disjoint union of the sets of roots of the $g_i$'s.

\begin{definition}
Given $f \in K[[X]][y]$ and a total order  $\preceq$ on $\mathbb{R}^n$ compatible with the group structure. A $\preceq$-branch of $f$ is the set of roots in $K_{\preceq}[[X^*]]$  of an irreducible element  $g \in K_\preceq[[X]][y]$ that divides $f$ as element of   $K_\preceq[[X]][y]$.
\end{definition}

The $\preceq_{(1, \sqrt{5})}$-branches of $f(y):=y^5+x_1^2x_2^2y^2+x_2^5$ are the sets $\{ \xi_1,\xi_2,\xi_3\}$ and $\{ \xi_4,\xi_5\}$. (This will become clear after Proposition \ref{el semigrupo para misma rama}.)

The field $K_{\preceq}((X^{\frac{1}{k}}))=K_{\preceq}((X))[X^\frac{1}{k}]$ is the splitting field of the separable polynomial $(y^k-x_1)(y^k-x_2)\cdots (y^k-x_n)$. Therefore it is a finite Galois extension of $K_{\preceq}((X))$. The elements of the Galois group of this extension are given by
\begin{equation}\label{985000}
\tau_{\eta,\mu}: \varphi(x_1^\frac{1}{k}, \ldots, x_n^\frac{1}{k}) \mapsto \varphi(\eta^{\mu_1}x_1^\frac{1}{k}, \ldots, \eta^{\mu_n}x_n^\frac{1}{k})
\end{equation} 
 where $\eta$ is a $k$-th primitive root of unity and $\mu =(\mu_1, \ldots, \mu_n) \in \{0, \ldots, k-1\}^n$.

\begin{proposition}\label{el semigrupo para misma rama}
    Given $f \in K[[X]][y]$ and a total order  $\preceq$ on $\mathbb{R}^n$, compatible with the group structure. Let $\xi$ be a root of $f$ in $K_{\preceq}((X^*))$, let $k$ be such that $f$ splits in $K_\preceq((X^{\frac{1}{k}}))[y]$ and let $\eta$ be a primitive $k$-th root of unity. The $\preceq$-branch of $f$ containing  $\xi$ is the set
    \[
    B(\xi)= \{\tau_{\eta,\mu} (\xi ) ;  \mu =(\mu_1, \ldots, \mu_n) \in \{0, \ldots, k-1\}^n\}
    \]
    where $\tau_{\eta,\mu}$ is as in (\ref{985000}).
\end{proposition}

\begin{proof}
Let $g \in K_{\preceq}[[X]][y]$ be an irreducible element that divides $f$ and such that $g(\xi)=0$. The elements of the $\preceq$-branch of $f$ containing $\xi$ are the roots of $g$ in $K_{\preceq}((X^*))$.

The morphisms $\tau_{\eta,\mu} : K_{\preceq}[[X^\frac{1}{k}]] \longrightarrow K_{\preceq}[[X^{\frac{1}{k}}]]$ respect the ring structure; hence
$$
g(\tau_{\eta,\mu}(\xi))=\tau_{\eta,\mu}(g(\xi))=0.
$$
Therefore $\tau_{\eta,\mu}(\xi)$ is in the $\preceq$-branch of $f$ containing $\xi$.

Let $L$ be the splitting field of $g$ and let $H$ be the Galois group of the extension $L/K_{\preceq}((X))$.

Since $g$ is irreducible in $K_{\preceq}[[X]][y]$, by Lemma \ref{irreducibilidad}, it is irreducible in $K_{\preceq}((X))[y]$. Since $L$ is a splitting field for $g$, the group $H$ acts transitively on the roots of $g$.

Moreover, since $K_{\preceq}((X)) \subset L \subset K_{\preceq}((X^\frac{1}{k}))$, for each $\tau \in H$, there exists $\vartheta \in \operatorname{Gal}(K_{\preceq}((X^\frac{1}{k}))/K_{\preceq}((X)))$ such that $\tau=\vartheta|_L$. The result follows from the fact that the elements of $\operatorname{Gal}(K_{\preceq}((X^\frac{1}{k}))/K_{\preceq}((X)))$ are given by (\ref{985000}).

\end{proof}

\begin{example}
    The polynomial $f(z)=z^4-4x_2z^3+\left(6x_2^2-2x_2^3-2x_1x_2^4\right)z^2+\left(-4x_2^3+4x_2^4+4x_1x_2^5\right)z+x_2^4-2x_2^5+x_2^6-2x_1x_2^6-2x_1x_2^7+x_1^2x_2^8 \in \mathbb{C}[[x_1,x_2]][z]$, as in Example~\ref{essd}, has a unique $\preceq$-branch for any order $\preceq$ compatible with the first orthant.
\end{example}

\begin{example}

The polynomial $f(y):=y^4 - 2(x_1+x_2)y^2 + x_1^2 - 2x_1x_2 + x_2^2 \in \mathbb{C}[[x_1,x_2]][y]$, as in Example \ref{23231s}, has only one $\preceq$-branch,
for any order $\preceq$ compatible with the first orthant.

\end{example}

\begin{example}
    The polynomial $f(z)= z^2-(x+y^2) \in \mathbb{C}[[x,y]][z]$, as in Example \ref{ejemplo3}, has only one
    $\preceq_{(1, \sqrt{2})}$-branch whilst it has two $\preceq_{(4, \sqrt{2})}$-branches.
\end{example}

\section{The family of semigroups of values of a hypersurface singularity}

A valuation on a field $K$ is a mapping $\nu:K^* \rightarrow G$, where $(G,\preceq)$ is a totally ordered abelian group and, for $a, b\in K^*$, $\nu(ab)=\nu(a)+\nu(b)$ and  $\nu(a+b)\succeq \min_\preceq\{\nu(a), \nu(b)\}$. The image of a valuation is a subgroup of $G$. The mapping $ord_x$ in (\ref{orden de serie}) is an example of a valuation.

Note that
$$K_{\preceq}((X^*))=\left\{\varphi \ ; \ \exists \  \sigma \subset(\mathbb{R}^n)_{\succeq \underline{0}}, \  \gamma \in \mathbb{Z}^n \,\text{ and }\, k\in \mathbb{N} \,\text{ with }\, Supp(\varphi)\subset (\gamma +\sigma) \cap \frac{1}{k}\mathbb{Z}^n\right\},$$ therefore, $K_{\preceq}((X^*))$ has a natural valuation with value group $(\mathbb{Q}^n, \preceq )$ given by:
$$\nu_\preceq (\varphi ):= \min_\preceq Supp(\varphi).$$
The above definition extends the definition of order in (\ref{orden de serie}). Also, $\nu_\preceq (\varphi )$ is the $\preceq$-dominant exponent of $\varphi$.

Note that, since $Supp(\varphi )=Supp( \tau_{\eta,\mu}(\varphi))$, for $\varphi\in K_{\preceq}((X^{\frac{1}{k}}))$ and  $\tau_{\eta,\mu}$ as in (\ref{985000}), we have
\begin{equation}\label{NoCambiaConGrupoGal}
    \nu_\preceq (\varphi )=\nu_\preceq (\tau_{\eta,\mu}(\varphi ) ).
\end{equation}
That is, the elements of the Galois group of the extension $K_{\preceq}((X))\subset K_{\preceq}((X^{\frac{1}{k}}))$ are automorphisms of the valued field $(K_{\preceq}((X^{\frac{1}{k}})), \nu_\preceq)$.

An element $\xi\in K_{\preceq}((X^*))$ induces a mapping that extends $\nu_\preceq $ to $K_{\preceq}((X^*)) [y]\setminus (y-\xi)$ given by
$$\nu_\preceq ^\xi (P(y)):= \nu_\preceq (P (\xi))$$
where $(y-\xi)$ denotes the ideal of $K_{\preceq}((X^*)) [y]$ generated by $y-\xi$.

Given a subring $\mathcal{A}\subset K_\preceq[[X]]$ the subset of $\mathbb{Q}^n$ 
$$\Gamma^\preceq_{\xi ,\mathcal{A}}:= \{ \nu_\preceq (h(\xi )) \ ; \ h\in \mathcal{A}[y]\setminus  (y-\xi)\}$$
is a semigroup.

The semigroup considered by A. Sathaye is obtained by taking the reverse lexicographical order as $\preceq$. When $f$ is a quasi-ordinary polynomial, and $\mathcal{A}=K[[x_1,\ldots , x_n]]$ the semigroup $\Gamma^\preceq_{\xi ,\mathcal{A}}$ does not depend on the order $\preceq$ and is the semigroup studied by P. D. González. For a fixed cone $\sigma$, the semigroups considered by A. Ali and A. Assi  for a $\sigma $-free singularity are obtained by taking $\mathcal{A}= K_\sigma [[X]]$ and $\preceq$ compatible with $\sigma$. 

Although it is not explicitly stated in their paper, the construction of A. Ali and A. Assi depends both on the chosen cone $\sigma$ and on the chosen order.

If $f$ is $\sigma$-free and the discriminant of $f$ with respect to $y$ equals a monomial times a unit in $K_\sigma [[X]]$ then $f$ is an irreducible polynomial in $K_\sigma [[X]][y]$ and defines a toric quasi-ordinary branch (see \cite{Gonzalez:2000}). In this case, $\Gamma^\preceq_{\xi ,\mathcal{A}}$ does not depend on the order $\preceq$ as long as $\sigma\subset {(\mathbb{R}^n)}_{\succeq \underline{0}}$ and is the semigroup of a toric quasi-ordinary branch studied in \cite{Gonzalez:2003b}.

\begin{theorem}
    Given an order $\preceq$ and a subring $\mathcal{A}\subset  K_{\preceq}((X))$.  If $\xi$ and $\xi'$ are roots of $f \in K[[X]][y]$ in  $K_{\preceq}((X^*))$ that belong to the same branch, then $\Gamma^\preceq_{\xi ,\mathcal{A}}=\Gamma^\preceq_{\xi' ,\mathcal{A}}$.
\end{theorem}
\begin{proof}
Given $f \in K[[X]][y]$ and a total order  $\preceq$ on $\mathbb{R}^n$, compatible with the group structure. Let $\xi$ and $\xi'$ be roots  of $f$ in $K_{\preceq}((X^*))$ in the same $\preceq$-branch, let $k$ be such that $f$ splits in $K_\preceq[[X^{\frac{1}{k}}]][y]$. By Proposition \ref{el semigrupo para misma rama}, there exists $\eta$, a primitive $k$-th root of unity, and $ \mu =(\mu_1, \ldots, \mu_n) \in \{0, \ldots, k-1\}^n$ such that
    \[
    \xi'= \tau_{\eta,\mu} (\xi )
    \]
    where $\tau_{\eta,\mu}$ is as in (\ref{985000}).

    Given $h\in \mathcal{A}[y]\subset K_{\preceq}((X)) [y]$, we have $h(\xi')=h(\tau_{\eta,\mu} (\xi ))= \tau_{\eta,\mu} (h(\xi))$.
    
    Now, for any $\varsigma\in K_\preceq[[X^{\frac{1}{k}}]]$, $Supp (\varsigma)= Supp (\tau_{\eta,\mu}(\varsigma ))$, and then $\nu_\preceq (\varsigma)= \nu_\preceq(\tau_{\eta,\mu}(\varsigma ))$. This implies that
    $$\nu_\preceq (h(\xi))= \nu_\preceq (h(\xi')).$$
\end{proof}

\begin{corollary}
    Given an order $\preceq$ and a subring $\mathcal{A}\subset  K_{\preceq}((X))$.  If $\xi$ and $\xi'$ are roots of $f \in K[[X]][y]$ and $f$ is irreducible as an element of  $K_{\preceq}((X))[y] $, then $\Gamma^\preceq_{\xi ,\mathcal{A}}=\Gamma^\preceq_{\xi' ,\mathcal{A}}$.
\end{corollary}

\section{Examples}

In this section, we compute several semigroups $\Gamma^\preceq_{\xi,\mathcal{A}}$ in order to illustrate how they vary with the choice of order $\preceq$ and of the ring $\mathcal{A} \subset K_\preceq((X))$. In particular, throughout this section, the semigroups are computed with respect to the ring $\mathcal{A}=K_\sigma[[X]]$.

The following proposition will be very helpful for computations.

\begin{proposition}\label{gradoh}
Given a monic polynomial $f \in K[[X]][y]$, a root $\xi \in K_\preceq((X^*))$ of $f$, and a ring $\mathcal{A} \subset K_\preceq((X))$, it follows that
\[
\Gamma^\preceq_{\xi,\mathcal{A}}
=
\left\{
\nu_\preceq(h(\xi)) \; ; \; h \in \mathcal{A}[y], \ \partial_y(h)<\partial_y(f)
\right\},
\]
where $\partial_y(P)$ denotes the degree of $P$ in $y$.
\end{proposition}
\begin{proof}

Given $h \in \mathcal{A}[y]$, there exist $c,r \in \mathcal{A}[y]$ with $\partial_y(r)<\partial_y(f)$ such that
\[
h=cf+r.
\]
Evaluating at $\xi$, we obtain
\[
h(\xi)=c(\xi)f(\xi)+r(\xi)=r(\xi),
\]
since $f(\xi)=0$.
\end{proof}

\begin{example}
 
Take $\xi_1$ to be the root of $f$ in Example \ref{essd}, $\xi_1:=x_1^{\frac{1}{2}}x_2^{\frac{3}{2}}+x_1^{\frac{3}{2}}x_2^{\frac{5}{2}}-3x_1^{\frac{3}{2}}x_2^{\frac{7}{2}}+x_1^2x_2^4
$$, 
$
that belongs to $\mathbb{C}[[x_1^{\frac{1}{2}}, x_2^{\frac{1}{2}}]]$.

For any order $\preceq$ we have
\[
\nu_\preceq(\xi_1)=\left(\frac{1}{2},\frac{3}{2}\right)
\]
and then,
\[
\mathbb{Z}^2\cap \sigma
+
\left(\frac{1}{2},\frac{3}{2}\right)\mathbb{Z}_{\geq 0} \subset \Gamma^\preceq_{\xi_1,\mathbb{C}_\sigma[[x_1,x_2]]}
\]
for any cone $\sigma$ compatible with $\preceq$.

On the other hand, by Proposition~\ref{gradoh}, for any $\nu\in\Gamma^\preceq_{\xi_1,\mathbb{C}_\sigma[[x_1,x_2]]}$ there exists a polynomial of degree less than 2,
\[
h=r_0+r_1z,
\qquad \text{with } r_0,r_1 \in \mathbb{C}_\sigma[[x_1,x_2]]
\]
such that $\nu= \nu_\preceq (h(\xi_1))$.

We have
\[
\nu_\preceq(r_0)\in \mathbb{Z}^2\cap\sigma \qquad\text{and}\qquad \nu_\preceq(r_1\xi_1)\in \left(\frac{1}{2},\frac{3}{2}\right) + \mathbb{Z}^2\cap\sigma
\]
then
\[
\nu_\preceq(r_0)\neq \nu_\preceq(r_1\xi_1)
\]
which implies
\[
\nu=\nu_\preceq(r_0)\quad\text{or}\quad \nu=\nu_\preceq(r_1\xi_1).
\]

Therefore,
\[
\Gamma^\preceq_{\xi_1,\mathbb{C}_\sigma[[x_1,x_2]]}
=
\mathbb{Z}^2\cap \sigma
+
\left(\frac{1}{2},\frac{3}{2}\right)\mathbb{Z}_{\geq 0}.
\]

Note that, 
\[
\Gamma^\preceq_{\xi_1 ,\mathbb{C}[[X]]}\subset {\left({\mathbb{R}_{\geq 0}}\right)}^n
\]
and, for different orders $\preceq$ and $\preceq'$, and any cone $\sigma$, compatible with $\preceq$ and $\preceq'$, we have the equality
\[\Gamma^\preceq_{\xi_1 ,\mathbb{C}_\sigma[[X]]}=\Gamma^{\preceq'}_{\xi_1 ,\mathbb{C}_\sigma[[X]]}=\Gamma^\preceq_{\xi_1 ,\mathbb{C}[[X]]} + \left( \sigma \cap \mathbb{Z}^2\right).
\]

\end{example}

\begin{example} Let
$f$ be
as in Example~\ref{23231s}, and let $\xi$ be the root
$
\xi:= x_1^{\frac{1}{2}} + x_2^{\frac{1}{2}} \in \mathbb{C}[[x_1^{\frac{1}{2}},x_2^{\frac{1}{2}}]].
$

Let $\preceq$ be an order, and let $\sigma$ be a cone compatible with it.

\begin{enumerate}
    \item[Case (1):]  The order $\preceq$ is compatible with  $\eta_1 := \langle (1,0), (-1,1)\rangle$.

    Consider
    $h_0:=z$ and $h_1:=z^2-x_1.$

    We have,
    $\nu_\preceq\left(h_0(\xi)\right)=
    \nu_\preceq\left(x_1^{\frac{1}{2}} + x_2^{\frac{1}{2}}\right)
    =
    \min_\preceq\left\{
    \left(\frac{1}{2},0\right),
    \left(0,\frac{1}{2}\right)
    \right\},
    $
    since
    \[
    \left(0,\frac{1}{2}\right)-\left(\frac{1}{2},0\right)
    =
    \left(-\frac{1}{2},\frac{1}{2}\right)\in \eta_1 \subset \left(\mathbb{R}^2\right)_{\succeq \underline{0}},
    \]
    it follows that
    \begin{equation}
    \left(0,\frac{1}{2}\right)\succeq \left(\frac{1}{2},0\right)
    \quad \text{and} \quad
    \nu_\preceq\bigl(x_1^{\frac{1}{2}} + x_2^{\frac{1}{2}}\bigr)
    =
    \nu_\preceq\bigl(h_0(\xi)\bigr)
    =
    \left(\frac{1}{2},0\right).
    \label{may}
    \end{equation}

    Similarly,
    \[
    \nu_\preceq\bigl(h_1(\xi)\bigr)
    =
    \nu_\preceq\bigl(\xi^2-x_1\bigr)
    =
    \nu_\preceq\left(2x_1^{\frac{1}{2}}x_2^{\frac{1}{2}}+x_2\right)
    =
    \min_\preceq\left\{
    \left(\frac{1}{2},\frac{1}{2}\right),
    (0,1)
    \right\}.
    \]
    Since
    \[
    (0,1)-\left(\frac{1}{2},\frac{1}{2}\right)
    =
    \left(-\frac{1}{2},\frac{1}{2}\right)\in \eta_1 \subset \left(\mathbb{R}^2\right)_{\succeq \underline{0}},
    \]
    it follows that
    \begin{equation}
    (0,1)\succeq \left(\frac{1}{2},\frac{1}{2}\right)
    \quad \text{and} \quad
    \nu_\preceq\bigl(h_1(\xi)\bigr)
    =
    \left(\frac{1}{2},\frac{1}{2}\right).
    \label{may2}
    \end{equation}

Let $h \in \mathbb{C}_{\sigma}[[x_1,x_2]][y]$ be as in Proposition \ref{gradoh}, with $\partial_y(h)<\partial_y(f)$, and write
\[
h=r_{00}+r_{10}h_0+r_{01}h_1+r_{11}h_0h_1,
\]
where $r_{00}, r_{10}, r_{01}, r_{11} \in \mathbb{C}_{\sigma}[[x_1,x_2]]$.

    Evaluating $h$ at $\xi$ and using \eqref{may} and \eqref{may2}, we have
    \[
    \begin{aligned}
    \nu_\preceq(r_{00}) &\in \mathbb{Z}^2 \cap \sigma, \\
    \nu_\preceq\bigl(r_{10}h_0(\xi)\bigr) &\in \left(\frac{1}{2},0\right)+\mathbb{Z}^2 \cap \sigma, \\
    \nu_\preceq\bigl(r_{01}h_1(\xi)\bigr) &\in \left(\frac{1}{2},\frac{1}{2}\right)+\mathbb{Z}^2 \cap \sigma, \\
    \nu_\preceq\bigl(r_{11}h_0h_1(\xi)\bigr) &\in \left(1,\frac{1}{2}\right)+\mathbb{Z}^2 \cap \sigma.
    \end{aligned}
    \]

    We conclude that
    \[
    \nu_\preceq\bigl(r_{ij}h_0^ih_1^j\bigr)
    \neq
    \nu_\preceq\bigl(r_{i'j'}h_0^{i'}h_1^{j'}\bigr)
    \quad \text{whenever } (i,j)\neq(i',j').
    \]
    Hence, there exists $(i,j)\in \{0,1\}^2$ such that
    \[
    \nu_\preceq\bigl(h(\xi)\bigr)=\nu_\preceq\bigl(r_{ij}h_0^ih_1^j\bigr).
    \]

    Finally, we have
    \[
    \Gamma^\preceq_{\xi,\mathbb{C}_{\sigma}[[x_1, x_2]]}=\mathbb{Z}^2\cap \sigma
    +
    \left(\frac{1}{2},0\right)\mathbb{Z}_{\geq 0}
    +
    \left(\frac{1}{2},\frac{1}{2}\right)\mathbb{Z}_{\geq 0}.
    \]

    \item[Case (2):] The order $\preceq$ is not compatible with $\eta_1$.  
    
    Since $\preceq$ is not compatible with $\eta_1$, it follows that $(-1,1)\prec(0,0)$. By additivity of $\preceq$, adding $(1,-1)$ to both sides yields $(0,0)\prec(1,-1)$. Since, moreover, $(0,0)\preceq(0,1)$, we conclude that $\preceq$ is compatible with $\eta_2:=\langle(0,1),(1,-1)\rangle.$
 
Consider
$h_0=z$ and $h_1=z^2-x_2$.

Proceeding as in Case (1), we obtain
\[
\nu_\preceq\bigl(h(\xi)\bigr)
\in
\mathbb{Z}^2\cap \sigma
+
\left(0,\frac{1}{2}\right)\mathbb{Z}_{\geq 0}
+
\left(\frac{1}{2},\frac{1}{2}\right)\mathbb{Z}_{\geq 0}
=
\Gamma^\preceq_{\xi,\mathbb{C}_{\sigma}[[x_1, x_2]]}.
\]

\end{enumerate}

Note that, as in the preceding example, for any cone $\sigma$ compatible with $\preceq$, we have

\begin{equation}\label{Creo que pasa para las libres}
\Gamma^\preceq_{\xi_1 ,\mathbb{C}[[X]]}\subset {\left({\mathbb{R}_{\geq 0}}\right)}^n \qquad \text{and}\qquad\Gamma^\preceq_{\xi_1,\mathbb{C}_\sigma[[X]]}=\Gamma^\preceq_{\xi_1,\mathbb{C}[[X]]}+(\sigma\cap\mathbb{Z}^2).
\end{equation}

However, in this example, for two different orders $\preceq$ and $\preceq'$ compatible with $\sigma$, the semigroups $\Gamma^\preceq_{\xi_1,K_\sigma[[X]]}$ and $\Gamma^{\preceq'}_{\xi_1,K_\sigma[[X]]}$ may be different.

\end{example}

\begin{example} 
Let  $f(z)= z^2-(x+y^2) \in \mathbb{C}[[x,y]][z]$, and let $\eta := \langle (0,1), (1,-2)\rangle$ and $\Tilde{\eta} = \langle (1,0), (-1,2)\rangle$ (as in Example \ref{ejemplo3}).
\begin{enumerate}
    \item
    Take $\Tilde{\xi}_1$ to be the root of $f$ in $K_{\Tilde{\eta}}[[x^{\frac{1}{2}},y^{\frac{1}{2}}]]$ given by 
    \[
    \Tilde{\xi}_1 = \sum_{i\geq 0}
    \begin{pmatrix}
        \frac{1}{2}\\
        i
    \end{pmatrix}
    x^{\frac{1}{2}-i}y^{2i}.
    \]
Let $\preceq$ be an order compatible with $\Tilde{\eta}$, and let $\sigma$ be a cone compatible with $\preceq$.
   
    Consider $h=r_0(x,y)+r_1(x,y)z$ with $r_i\in K_{\sigma}[[x,y]]$.

    Then, 
    \[
    \nu_1:=\nu_\preceq(r_0(x,y)) \in \mathbb{Z}^2\cap \sigma
    \]
    and, since $\nu_\preceq(\Tilde{\xi}_1)=\left(\frac{1}{2},0\right)$,
    \[
    \nu_2=\nu_\preceq(r_1(x,y)\Tilde{\xi}_1)\in \left(\frac{1}{2},0\right)+\mathbb{Z}^2\cap \sigma.
    \]

    Then $\nu_1\neq \nu_2$, since the first coordinate of $\nu_1$ is an integer, whereas the first coordinate of $\nu_2$ is not. Therefore,
    \[
    \nu_\preceq(h(\Tilde{\xi}_1))=\nu_\preceq(r_0)
    \qquad \text{or} \qquad
    \nu_\preceq(h(\Tilde{\xi}_1))=\nu_\preceq(r_1\Tilde{\xi}_1).
    \]
    Hence,
    \[
    \nu_\preceq(h(\Tilde{\xi}_1))
    \in
    \mathbb{Z}^2\cap \sigma
    +
    \left(\frac{1}{2},0\right)\mathbb{Z}_{\geq 0}
    =
    \Gamma^\preceq_{\Tilde{\xi}_1,\mathbb{C}_{\sigma}[[x,y]]}.
    \]

Note that, as in the preceding examples, for any cone $\sigma$ compatible with $\preceq$, we have

\[
\Gamma^\preceq_{\Tilde{\xi}_1 ,\mathbb{C}[[X]]}\subset {\left({\mathbb{R}_{\geq 0}}\right)}^n \qquad \text{and}\qquad\Gamma^\preceq_{\Tilde{\xi}_1,\mathbb{C}_\sigma[[X]]}=\Gamma^\preceq_{\Tilde{\xi}_1,\mathbb{C}[[X]]}+(\sigma\cap\mathbb{Z}^2).
\]

\item 
Take $\xi_1$ to be the root of $f$ in $K_{\eta}[[x,y]]$ given by 
    \[
    \xi_1 = \sum_{r\geq 0}
    \begin{pmatrix}
        \frac{1}{2}\\
        r
    \end{pmatrix}
    x^ry^{1-2r},
    \]
let $\preceq$ be an order compatible with $\eta$, and let $\sigma$ be a cone compatible with $\preceq$. 
\begin{enumerate}
    \item 
    If $\eta\subset\sigma$ then
$$
\sigma \cap \mathbb{Z}^2=\Gamma^\preceq_{\xi_1,\mathbb{C}_\sigma[[x,y]]}.
$$
    \item 
    If $\eta\subsetneq\sigma $ set 
 \[
    \varphi := \sum_{r\geq 0, (r,1-2r)\in\sigma}
    \begin{pmatrix}
        \frac{1}{2}\\
        r
    \end{pmatrix}
    x^ry^{1-2r}
    \]
    and set $h:= z-\varphi$. 

We have
\[
\nu_\preceq (h(\xi_1))\notin\sigma\quad\text{and}\quad\nu_\preceq (h(\xi_1))\in\Gamma^\preceq_{\xi_1,\mathbb{C}_\sigma[[x,y]]}
\]
so, in this case,
\[
\Gamma^\preceq_{\xi_1 ,\mathbb{C}_\sigma [[x,y]]}\nsubseteq \sigma.
\]
\end{enumerate}

\end{enumerate}
\end{example}

\section*{Acknowledgments}
We extend our gratitude to Dr. Fabiola Manjarrez Gutiérrez for her valuable comments and suggestions, which have significantly enhanced this work and to Adrien Poteaux and Martin Weimann who pointed out some mistakes in a previous version of this paper. We would like to thank Patrick Popescu-Pampu for providing us with the references by Arf and Du Val.
We also thank both referees for their comments and suggestions that significantly improved the article.

This project was supported by DGAPA-UNAM through PAPIIT IN113323 and PAPIIT IN112626.

Giovanna Ilardi was partially supported by GNSAGA-INDAM and by the Universit\`a degli Studi di Napoli, ``Lefschetz Properties, arrangements and unexpected hypersurfaces: geometric, algebraic and computational point of view", funded by the University, project FRA 2024 linea A --
CUP: E63C24002510005.

\bibliographystyle{plain}
\bibliography{referencias}

\end{document}